\documentclass[12pt]{article}
\usepackage{graphicx}
\usepackage{amsmath}
\usepackage{authblk}
\usepackage{amsthm}
\usepackage{amssymb}
\usepackage{xcolor}
\usepackage[colorlinks=true, linkcolor=black, citecolor=black, urlcolor=black]{hyperref}
\usepackage[margin=1in]{geometry}
\usepackage{array,booktabs}
\usepackage{amssymb, latexsym, xy, eucal, mathrsfs, graphicx, tikz, amsmath, caption, amsfonts, MnSymbol, lipsum, subfigure, cite, mathbbol}
\usepackage{lineno}
\renewcommand\thanks[1]{\protect\footnotetext[0]{#1}}
\newtheorem{theorem}{Theorem}[section]
\newtheorem{corollary}{Corollary}[theorem]
\newtheorem{lemma}[theorem]{Lemma}

\newtheorem{proposition}{Proposition}[section]
\theoremstyle{definition}
\newtheorem{definition}{Definition}[section]

\newcommand{\n}{\mathbb{1}}
\theoremstyle{remark}

\newtheorem{problem}{Problem}

\begin{document}
\sloppy
\title{On Structural and Spectral Properties of Distance Magic Graphs}
\date{}
\author[1]{Himadri Mukherjee\thanks{himadrim@goa.bits-pilani.ac.in}}
\author[1]{Ravindra Pawar\thanks{p20200020@goa.bits-pilani.ac.in }}
\author[1]{Tarkeshwar Singh\thanks{tksingh@goa.bits-pilani.ac.in}}
\affil[1]{Department of Mathematics,

BITS Pilani K K Birla Goa Campus, Goa, India.}

\maketitle
\noindent
\begin{abstract}
A graph $G=(V,E)$ is said to be distance magic if there is a bijection $f$ from a vertex set of $G$ to the first $|V(G)|$ natural numbers such that for each vertex $v$, its weight given by $\sum_{u \in N(v)}f(u)$ is constant, where $N(v)$ is an open neighborhood of a vertex $v$. In this paper, we introduce the concept of $p$-distance magic labeling and establish the necessary and sufficient condition for a graph to be distance magic. Additionally, we introduce necessary and sufficient conditions for a connected regular graph to exhibit distance magic properties in terms of the eigenvalues of its adjacency and Laplacian matrices. Furthermore, we study the spectra of distance magic graphs, focusing on singular distance magic graphs. Also, we show that the number of distance magic labelings of a graph is, at most, the size of its automorphism group.
\end{abstract}
\section{Introduction}
Throughout this article, we assume that $G = (V,E)$ denotes a simple, undirected, finite graph on $n$ vertices. Let $N(u) = \{v \in V : uv \in E\}$ be the neighborhood of a vertex $u \in V(G)$. We denote the degree of a vertex $v$ by $deg(v)$ and is given by $deg(v) = |N(v)|$. Let $K_n$ denote a complete graph on $n$ vertices. 

\smallskip
A bijective function $f: V(G) \rightarrow \{1,2, \dots, n\}$ is called the \textit{vertex labeling} of a graph $G$. If the sum $\sum_{u \in N(v)}f(u)$ is constant $k$, for all $v \in V$, then the vertex labeling $f$ is called a \textit{distance magic labeling}, $k$ is called a \textit{magic constant}. Any graph $G$ that admits the distance magic labeling is called a \textit{distance magic graph}. It is known that the magic constant of a graph $G$, if it exists, is unique, that is, independent of distance magic labeling \cite{uniquek, o_uniquek}. Since the general characterization of a graph as distance magic is not known, researchers have studied the existence of distance magic labeling for many families of graphs \cite{dm_survey_gallian}. Although there are numerous necessary conditions for a graph to possess a distance magic labeling \cite{vilfredt, uniquek, o_uniquek}, the elusive sufficient condition for this phenomenon remains undiscovered. In this paper, we give several necessary and sufficient conditions for a graph to be a distance magic.

\smallskip
A multiset is a set in which repetition of elements is allowed. We formally define the multiset of our interest.
\begin{definition}
Given an integer $p \ge 1$, by $\{1, 2, \dots, n \}_p$, we mean a \textit{multiset} obtained by reducing all numbers $1, 2, \dots, n$ to modulo $p$ and replacing $0$s if any by $p$.
\end{definition}
For example, with $p = 4$ and $n = 9$, we have the multiset $\{1, 2, 3, 4, 1, 2, 3, 4, 1\}_4$ obtained by reducing the elements $1, 2, 3, 4, 5, 6, 7, 8, 9$ to modulo $4$ and replacing $0$s by $4$. Note that for $p > n$, $\{1,2, \dots, n\}_p = \{1,2, \dots, n\}$.

\smallskip
We introduce the concept of $p$-distance magic labeling as a generalization of distance magic labeling of graphs.
\begin{definition}
Let $G$ be a graph on $n$ number of vertices and given an integer $p \ge 1$. Consider a multiset $T = \{1, 2, \dots, n\}_p$. We call graph $G$ a $p$-\textit{distance magic} if there is a bijective map $f$ from the vertex set $V$ to a multiset $T$ called a $p$\textit{-distance magic labeling} such that the weight of each vertex $x \in V$, denoted by $w(x)$, is equal to the same element $\mu (\bmod~p)$, where $w(x) = \sum_{y \in N(x)}f(y)$. 
\end{definition}
For example, consider a graph $G$ in Figure \ref{fig:2-magic} on $11$ vertices. With $p=2$, we have a multiset \(\{1,2,1,2,1,2,1,2,1,2,1\}_{2}\). Then, with labeling, as shown in the figure below, $G$ is $2$ -magic with magic constant $0$.

\begin{figure}[ht]
    \centering
    \includegraphics[scale=0.5]{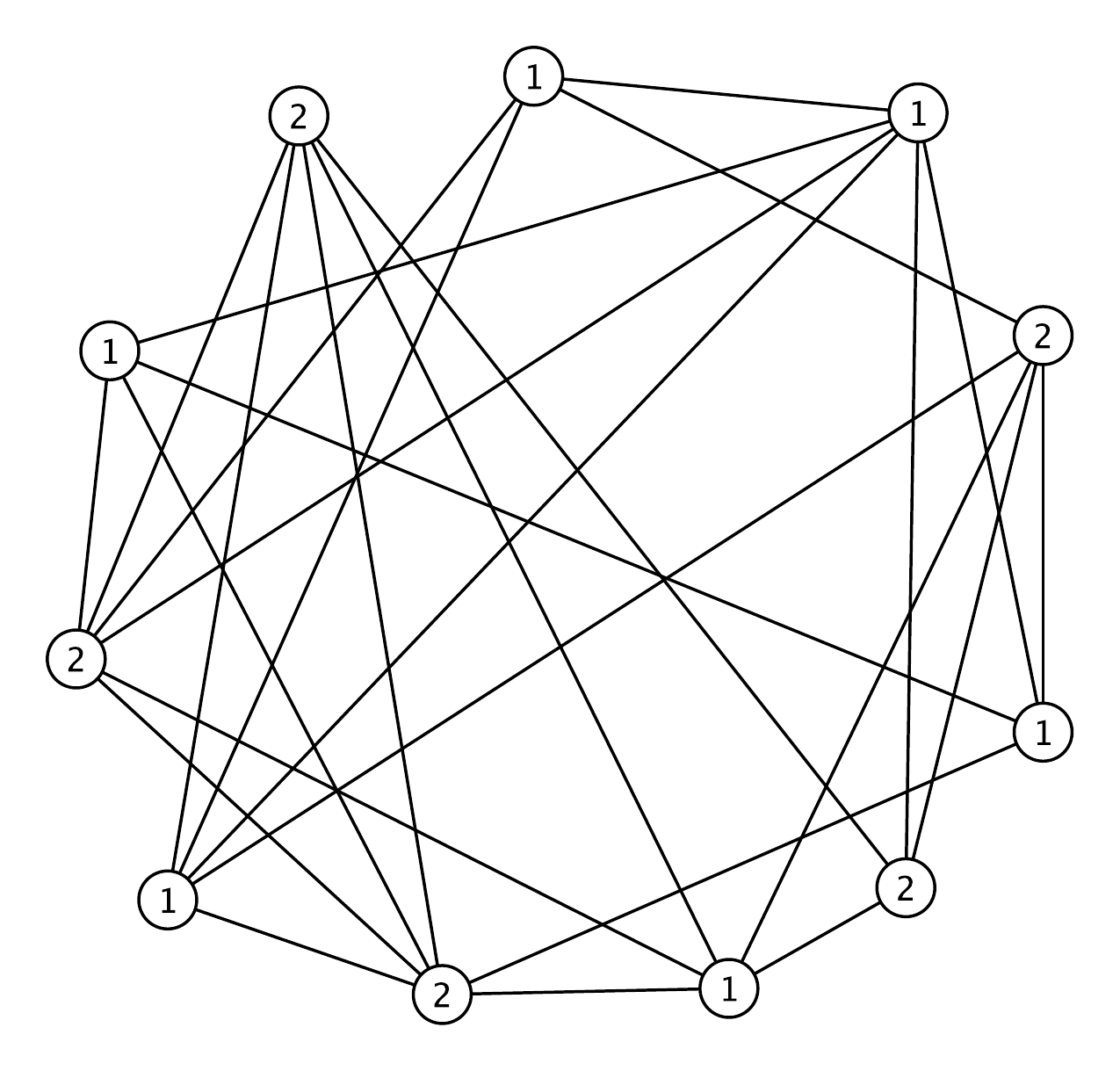}
    \caption{A $2$-distance magic labeling of a graph}
    \label{fig:2-magic}
\end{figure}

The $p$-distance magic labeling also constitute an infinite class of necessary conditions that become sufficient when true for all $p$. In the context of a graph with $n$ vertices, we simplify the labeling process by reducing the labels $1, 2, \dots, n$ modulo $p$, creating a multiset where the repetition of elements is allowed. We then calculate the weights modulo $p$ in the usual manner. This approach reduces the number of unique labels used, allowing the calculations to be streamlined. It is important to emphasize that $p$-distance magic labeling differs from group distance magic labeling or $\Gamma$-distance magic labeling \cite{group_cycles_dalibor}, particularly when $\Gamma = \mathbb{Z}_n$. In the former, the requirement for the number of distinct labels to equal the number of vertices is relaxed.

\smallskip
It is well known that there cannot be a subgraph avoidance criterion for the graphs that admit magic labeling (since any connected graph can be embedded in a distance magic graph \cite{sigma_jinnah}). As a result, the problem of magic labeling does not allow for the construction of magic labeling from the labeling on a subgraph. One can say that the problem of distance magic labeling is a difficult problem with little room for ``construction" in the underlying graph because it does not allow ``cutting and pasting" type techniques. However, $p$-distance magic labeling allows such cutting and pasting techniques, as is evident from applying the Chinese remainder theorem (see Theorem \ref{th:crt}). The simplification through the existence of a $p$-magic labeling and the patching up of two different such labeling through the Chinese remainder theorem gives us some constructive approach. Such possibilities give us hope that to tackle the problem of distance magic labeling, this ``construction" idea through $p$-distance magic labeling will prove to be very important. With sufficiently larger values of $p$, a $p$-distance magic labeling becomes the case of distance magic labeling, and with $p=n$, $p$-distance magic labeling becomes $\mathbb{Z}_n$-distance magic labeling.

\smallskip
We explore certain linear combinations of matrices associated with distance magic graphs to establish a condition that is necessary and sufficient for a graph to be considered distance magic. Specifically, we prove that if a graph $G$ possesses a distance magic property with a magic constant $k$, then $k$ is an eigenvalue of a particular linear transformation associated with $G$. Furthermore, a connected $r$-regular graph is distance magic if and only if $r^2$ is an eigen value of $L^2 + A^2$, where $A$ is an adjacency matrix and $L$ is the Laplacian matrix of $G$. Lastly, we obtain one necessary condition by proving that if a graph is distance magic, then $AA^{\dagger}$ is doubly stochastic, where $A$ is the adjacency matrix of $G$ and $A^{\dagger}$ is a Moore-Penrose inverse of $A$.

\smallskip
We have established bounds on the number of distance magic labelings of a distance magic graph based on an automorphism group of the graph. Moreover, we present several examples of distance magic graph where the number of distance magic labelings matches the cardinality of the automorphism group.

\smallskip
As mentioned earlier, forbidden subgraph characterisation is not possible for distance magic graphs. However, we have the existence of some subgraphs as a necessary condition for a graph to be distance magic. It can be proved easily using the following theorem.

\begin{theorem} \label{th: dm_symm} \cite{sigma_jinnah, vilfredt}
A graph $G$ is not distance magic if there are vertices $x$, $y$ in $G$ such that $\mid N(x) \triangle N(y) \mid \, = 1 \text{ or } 2$.
\end{theorem}
\begin{theorem}
Every distance magic graph contains $P_3$ or $C_4$ as a subgraph or $P_3$ as an induced subgraph. 
\end{theorem}
\begin{proof}
Let $G$ be a distance magic graph. Then by Theorem \ref{th: dm_symm}, for any two vertices $x$, $y$ in $G$ we must have $\mid N(x) \triangle N(y) \mid \, = 0$ or $\mid N(x) \triangle N(y) \mid \, \ge 3$. That is, for any two vertices in $G$, their neighborhoods are identical or symmetric difference of their neighborhoods is at least $3$. Consider the following two cases.

\noindent \textit{Case i.} Suppose there exists two vertices $x,y$ such that $N(x) = N(y)$. If $deg(x) = 1$ (and hence $deg(y) = 1$) then there is an unique vertex $z$ adjacent to both $x$ and $y$. Therefore, there is a path $P_3$ induced by vertices $x, y$, and $z$.

\noindent \textit{Case ii.} Suppose that $N(x) \ne N(y)$ for any two vertices $x, y$ in $G$. Then, for every pair of vertices $x, y$ in $G$, we have $\mid N(x) \triangle N(y) \mid \, \ge 3$. That is $\mid N(x) \cup N(y) \mid - \mid N(x) \cap N(y) \mid \, \ge 3$. Note that $\mid N(x) \cap N(y) \mid \, = 0$ for all $x, y \in V(G)$ is not possible. Otherwise, $G$ is a disjoint union of $K_2$ components, which is not distance magic. Therefore, there exist two vertices $x, y$ in $G$ such that $\mid N(x) \cap N(y) \mid \, \ge 1$. Let $\mid N(x) \cap N(y) \mid \, = 1$. If $N(x) \cap N(y) = \{z\}$ then $G$ contains $P_3$ as a subgraph. Note that $P_3$ need not be an induced subgraph on the vertices $x,y$, and $z$. If $\mid N(x) \cap N(y) \mid \, \ge 2$ then $G$ contains a $C_4$ (and hence $P_3$) as a subgraph.
\end{proof}

\smallskip
For a set $D \subset \mathbb{N}$, let $ND(v)=\{u \in V(G):d(u,v) \in D \}$, where $d(u,v)$ denotes a distance between vertices $u$ and $v$. Let $W$ be a multiset of real numbers. A graph $G$ is said to be $(D,W)$-vertex magic if there exists a bijection $g : V (G) \to W$ such that for all $v \in V(G)$, $\sum_{u \in N(v)} g(u)$ is a constant \cite{o_uniquek}. In \cite{s_algo_slater}, Slater proved that the problem of determining whether a graph is $(D,W)$-vertex magic is $NP$-complete. Also, it is known that this problem is $NP$-complete in special classes of graphs like $3$-regular graphs where labeling set is a multiset with entries from the set of positive integers \cite{s_algo_gupta}. This shows that the problem of distance magic labeling is $NP$-complete when labeling set is a multiset with entries from the set of positive integers.\\

\noindent \textbf{Organisation of paper:}

In Section \ref{sec: pmagic}, we present the results on $p$-distance magic labeling and the necessary and sufficient conditions for a graph to be considered distance magic in terms of $p$-distance magic labeling.
In section \ref{sec: matrices}, we show that there are infinitely many distance magic graphs having $0$ as an eigenvalue. Also, we derive the conditions for a graph to be considered distance magic based on eigenvalues of its adjacency and Laplacian matrices. In Section \ref{sec: action}, we analyse the action of the automorphism group of a graph on the set of its distance magic labelings. We estimate the lower bounds on the number of distance magic labelings for the graph.

\smallskip
For undefined terminology and notation related to graph theory, we refer to \cite{west}, and for the survey on distance magic labeling, we refer to \cite{dm_survey_gallian, dm_survey_arumugam}.

\section[p-distance magic graphs]{\texorpdfstring{\emph{p}-distance magic graphs}{p-distance magic graphs}} \label{sec: pmagic}

For a distance magic graph, $G$ of order $n$ with distance magic labeling $f$ and the magic constant $k$, the sum of the weights of all vertices is $nk$ (see \cite{dm_miller, vilfredt}). That is $\sum_{x \in V(G)} w(x) = \sum_{x \in V(G)}f(x)deg(x) = nk$. The maximum possible degree in $G$ is $n-1$. Therefore, we obtain $nk \le (n-1)\sum_{x \in V(G)}f(x) = \frac{(n-1)n(n+1)}{2}$. This implies, $k \le \frac{n^{2}-1}{2}$.

\begin{theorem} \label{main}
A graph $G$ is distance magic if and only if it is $p$-distance magic for infinitely many values of $p$.
\end{theorem}
\begin{proof}
Suppose that $G$ is a distance magic graph of order $n$ with magic constant $k$ and distance magic labeling $f$. Let $S$ be an infinite subset of those natural numbers $p$ for which $G$ is $p$-distance magic. Let $p \in S$ be arbitrary. Define $f_{p} : V \to \{1,2,\dots, n\}_p$ by $f_{p}(v) = f(v)(\bmod~{p})$. Therefore, the weight $w_{f_{p}}(x)$ of each vertex $x$ in $V$ is equal to the same number $k_{p} \equiv k (\bmod~{p})$. This proves the necessary part. Conversely, suppose that $G$ is $p$-distance magic for all $p \in S$. The set $S$ being infinite guarantees the existence of $q$ in $S$ such that $q > \frac{(n^{2}-1)}{2}$ and $G$ induce a $q$-distance magic labeling. Then $w(x) = k'$ for some $k' \in \mathbb{Z}_{q}$ and $k' \le \frac{(n^{2}-1)}{2} < q$. Therefore, $q$-distance magic labeling is the required distance magic labeling of $G$. This proves that $G$ is a distance magic graph.
\end{proof}
\begin{corollary}  
A graph $G$ on $n$ vertices is distance magic if and only if it is $p$-distance magic for all $p \ge 1$.
\end{corollary}
\begin{corollary}
A graph $G$ on $n$ vertices is distance magic if and only if it is $p$-distance magic for all prime numbers $p$.
\end{corollary}

It follows from the Theorem \ref{main} that the $p$-distance magic labeling coincides with distance magic labeling whenever $p$ is sufficiently large. This shows that $p$-distance magic labeling is a generalization of distance magic labeling. Also, some structural characterisations are possible with $p=2$. Since $p = 2$ is the least possible, we focus on $2$-distance magic labelings. \par Let $G$ be a graph whose vertices are labeled with the labeling $f$ using the numbers in multiset \(T = \{1,2, \dots, n\}_p\) for some $p \ge 1$. Let $V_i = \{v \in V : f(v) \equiv i(\bmod~{p})\}$, for all $i = 1,2, \dots, p$. By $G_i$, we mean a subgraph of $G$ induced on $V_i$.

Recall that a graph $G$ is \textit{Eulerian} if and only if it has at most one non-trivial component and its vertices all have even degrees, and a \textit{matching} in a graph is an independent subset of a set of edges.
\begin{theorem}
Let $G$ be a graph on $n$ vertices that admits $2$-distance magic labeling. If
\begin{enumerate}
    \item a magic constant is $0$, then the subgraph $G_1$ is a union of Eulerian components.
    \item a magic constant is $1$ then the subgraph $G_1$ contains a matching.
\end{enumerate}
\end{theorem}
\begin{proof}
Suppose that a graph $G$ is $2$-distance magic. We prove the theorem in the following two cases.\\
\textit{Case} $i$. Let magic constant be $0$. If $G$ itself is an Eulerian, there is nothing to prove. Suppose $G$ is not Eulerian. For a magic constant to be $0$, each vertex in $G_1$ must have an even number of neighbors in $V_1$ and any number of neighbors in $V_2$. Therefore, each vertex in $G_1$ is of even degree. Therefore, $G_1$ is a union of Eulerian components.\\
\textit{Case} $ii$. Let magic constant be $1$. Let $v \in V_1$. For $w(v)=1, v$ must be adjacent to the odd number of vertices in $V_1$. Now, we construct a matching, say $M$. Let $v_1 \in V_1$ such that $vv_1 \in E(G_1)$. We add this edge $vv_1$ to the matching $M$. We repeat the process, which must end after a finite number of steps since $G_1$ is a finite graph. Therefore, we obtain a maximal matching $M$. This proves the theorem.
\end{proof}

\begin{figure}[ht]
     \begin{center}
        \subfigure[$G_1$ is an Eulerian]{%
            \label{fig:2magic1}
            \includegraphics[width=0.4\textwidth]{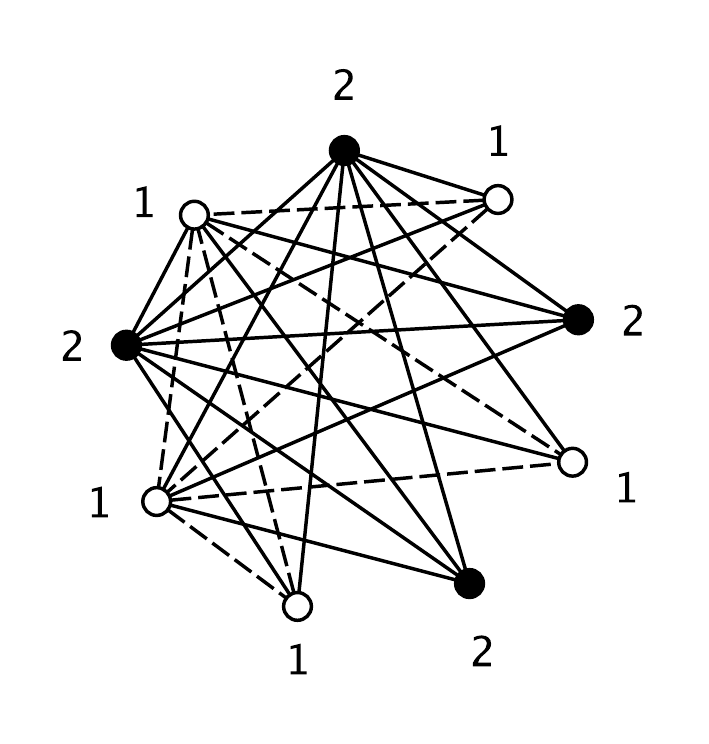}
            \label{2magic1}
        }%
        \subfigure[$G_1$ with matching]{%
           \label{fig:2magic2}
           \includegraphics[width=0.35\textwidth]{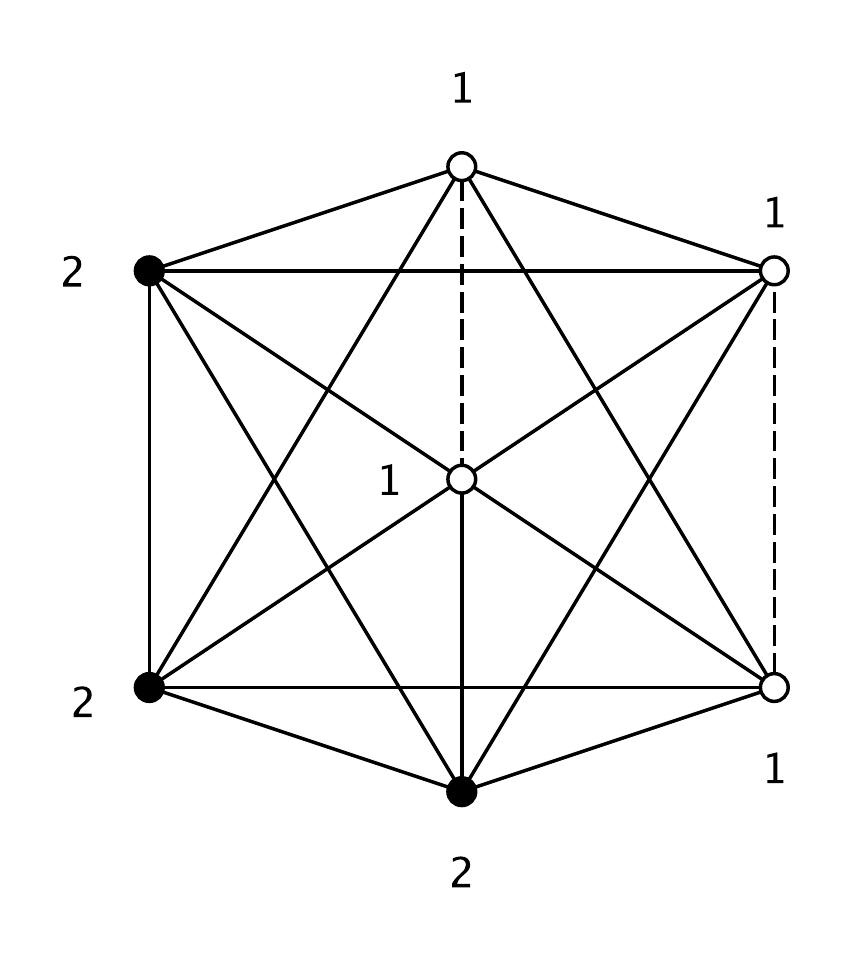}
        }%
    \end{center}
    \caption{%
        $2$-magic graphs with an Eulerian subgraph or a matching.
     }%
   \label{fig:2magicgraphs}
\end{figure}

Figure \ref{fig:2magicgraphs} shows the $2$-distance magic labeling of two graphs. A graph in Figure \ref{fig:2magic1} has magic constant $0$, and the subgraph $G_1$ induced by $V_1$ is depicted in dotted edges and hollow vertices is an Eulerian subgraph. A graph in Figure \ref{fig:2magic2} has a magic constant $1$, and $G_1$ contains a matching depicted in dotted edges. Note that in Figure \ref{fig:2magic2}, matching in $G_1$ is not unique, but it is a perfect matching. However, it is not the case that we always obtain a perfect matching. Also, in Figure \ref{fig:2magic2}, none of the vertices in $G_1$ are of even degree. Therefore, obtaining an Eulerian circuit is not possible.

\smallskip
Now we introduce the concept \textit{modulo regularity} of graphs and use to construct different $p$-distance magic labelings from the given $p$-distance magic labeling of a graph.

\begin{definition}
We call graph $G$ an \textit{$r$-modulo $p$ regular} if $deg(u) \equiv r (\bmod~{p})$ and we write $G$ is \textit{$r(\bmod~p$)-regular}.
\end{definition}

\begin{theorem} \label{multiple}
If $G$ is $r(\bmod~p)$-regular, $p$-distance magic graph on $n$ vertices with magic constant $k$ and $p$-distance magic labeling $f$, then for each positive integer $i$, the map $f' = f + i$ is also a $p$-distance magic labeling with magic constant $k' = (k+ir) (\bmod~{n})$.
\end{theorem}

\begin{proof}
Let $G$ be $r (\bmod~p)$-regular, $p$-distance magic graph with magic constant $k$ and labeling $f$ and let $i \in \mathbb{Z}^+$. Define $p$-distance magic labeling $g$ by $g(u) = i + f(u)$. Then, for $u \in V$, 
\begin{align*}
w(u) &= \sum_{uv \in E} g(v)\\
     &= \sum_{uv \in E} (i + f(v))\\
     &= i~deg(u) + \sum_{uv \in E} f(v)\\
     &\equiv (ir + k)(\bmod~p).
\end{align*}
This completes the proof.
\end{proof}

Theorem \ref{multiple} shows that a $p$-magic constant need not be unique. Nevertheless, in some cases, one can obtain the uniqueness of the same as described in the following theorem. (The ring $\mathbb{Z}_n$ is considered with the usual addition and multiplication modulo $n$.)

\begin{theorem} \label{unique}
Let $G$ be a graph on $n$ vertices. If a graph $G$ is a $p$-distance magic with magic constant $k$ and if $\frac{n(n+1)}{2}$ is a unit in the ring $\mathbb{Z}_p$, then $k$ is unique.
\end{theorem}

\begin{proof}
Let $G$ be a graph on $n$ vertices $\{x_1, x_2, \dots, x_n\}$ having two $p$-distance magic labelings $f$ and $g$ with respective magic constants $k$ and $l$. Let $\n$ be the column vector with all $n$ entries $1$. Let $A$ be an adjacency matrix of $G$ and put $X = (f(x_1), \dots, f(x_n))^\top$ and $Y = (g(x_1), \dots, g(x_n))^\top$. Since $f$ and $g$ are distance magic labelings with magic constants $k$ and $l$, it follows that $AX = k\n$ and $AY = l\n$. Since, $X^\top AY$ is $1 \times 1$ matrix, we have $X^\top AY = (X^\top AY)^\top = Y^\top AX$. This gives
\begin{align*}
 l X^\top \n &= k Y^\top \n\\
\implies l(1+2+ \dots + n) &= k(1+2+ \dots + n).
\end{align*}
Since $(1+2+ \dots + n)$ is a unit in $\mathbb{Z}_p$, by the cancellation laws we have $l = k$.
\end{proof}
 
\begin{corollary} \label{uniquecor1}
If $G$ is a $2$-distance magic graph of order $n \ge 3$, where $n \equiv 1 \text{ or } 2 (\bmod~{4})$ with magic constant $k$, then $k$ is unique.
\end{corollary}
\begin{proof}
Let $G$ be $2$-distance magic graph of order $n \ge 3$. Suppose $n \equiv 1 (\bmod~{4})$, that is, $n =4j+1$, for some integer $j$. Then $\frac{n(n+1)}{2} = (4j+1)(2j+1)$ which is unit in $\mathbb{Z}_2$. Similarly, when $n = 4j+2$, for some integer $j$, then $\frac{n(n+1)}{2} = (2j+1)(4j+3)$ which again an unit in $\mathbb{Z}_2$. This completes the proof.
\end{proof}

\begin{theorem} \label{th:crt}
If $G$ is a $p$-distance magic as well as $q$-distance magic graph on $n$ vertices for some relatively prime integers $p$ and $q$ such that $pq \le n$, then $G$ is $pq$-distance magic.
\end{theorem}
\begin{proof}
Let $G$ be a graph with vertex set $\{x_1, x_2, \dots, x_n\}$ that is both $p$-distance magic and $q$-distance magic for some relatively prime integers $p$ and $q$ such that $pq \le n$. Let $f_p, f_q$ be a $p$-distance and $q$-distance magic labelings, and $k_p, k_q$ be corresponding magic constants. For each $1 \le i \le n$, let $f_p(x_i) \equiv a_i (\bmod~p)$ and $f_q(x_i) \equiv b_i (\bmod~q)$. Since $p$ and $q$ are co-prime, by the Chinese remainder theorem, the system of congruences
\begin{align*}
x \equiv a_i (\bmod~p)\\ x \equiv b_i (\bmod~q)
\end{align*}
has unique solution say $y_i (\bmod~pq)$, for each $i (1 \le i \le n)$. Now we define new labeling $f_{pq}$ by $f_{pq}(x_i) = y_i$, for each $i(1 \le i \le n)$. The uniqueness of the labels $y_i$s guarantees that $f_{pq}$ is a bijective map. Now we calculate the weight of a vertex $x_i \in V$ under $f_{pq}$. 
\begin{align*}
    w(x_i) &= \sum_{x_j \in N(x_i)}f_{pq}(x_j)\\
           &\equiv 
              \begin{cases}
               k_p (\bmod~p) \\
               k_q (\bmod~q).
              \end{cases}
\end{align*}
For each $1 \le i \le n$, again we solve the system of congruences 
\begin{equation} \label{sys:1}
\begin{aligned}
w(x_i) \equiv k_p (\bmod~p)\\
w(x_i) \equiv k_q (\bmod~q)
\end{aligned}
\end{equation}
by the Chinese remainder theorem, and we conclude that the system of congruences (\ref{sys:1}) has a unique solution, denoted by $k_{pq} (\bmod~pq)$. This proves that $f_{pq}$ is the required labeling and $G$ is a $pq$-distance magic graph.
\end{proof}

Given a graph $G$ on $n$ vertices which is both $p$-distance magic and $q$-distance magic for some relatively prime integers $p$ and $q$ such that $pq \le n$ then using the Theorem \ref{th:crt}, we can find $pq$-distance magic labeling for $G$. When $pq > n$, we need not always obtain a $pq$-distance magic graph. 
For example, consider a cycle $C_4 = v_1, v_2, v_3, v_4, v_1$. Define $2$-magic labeling of $C_4$ by $f_2(v_1) = 1$, $f_2(v_2) = 2$, $f_2(v_3) = 2$, $f_2(v_4) = 1$ and a $3$-magic labeling by $f_3(v_1) = 2$, $f_3(v_2) = 1$, $f_3(v_3) = 3$, $f_3(v_4) = 1$ as shown in Figure \ref{fig:crt}. In Figure \ref{fig:crt}, the graph $G_1$ shows the $2$-distance magic labeling of $C_4$ with magic constant $1$, and the graph $G_2$ shows a $3$-distance magic labeling with magic constant $2$ of the same graph $C_4$. For each $i (1 \le i \le 4)$, we solve the system
\begin{equation} \label{sys:2}
\begin{aligned}
x \equiv f_2(x_i) (\bmod~2)\\
x \equiv f_3(x_i) (\bmod~3)
\end{aligned}
\end{equation}
using the Chinese remainder theorem. For each $i (1 \le i \le 4)$, the unique solution of the system (\ref{sys:2}) is $2,4,6,1$ respectively. We label the vertices of $C_4$ using these solutions: $f_6(v_1) = 2$, $f_6(v_2) = 4$, $f_6(v_3) = 6$, $f_6(v_4) = 1$ as shown in the graph $G_3$ of Figure \ref{fig:crt}. Observe that $f_6$ is not a map from $V(C_4)$ to $\{1,2,3,4\}_6$. Thus, in this case, we cannot obtain the $6$-distance magic graph using the given $2$-distance magic and $3$-distance magic graph. This does not contradict the Theorem \ref{th:crt} because $(p=2) \times (q=3) = 6 > 4 = n$. Therefore, condition $pq < n$ cannot be ignored in the statement of Theorem \ref{th:crt}. However, it may happen that in some cases, each label $y (\bmod~pq)$ obtained using the Chinese remainder theorem as described in the above theorem satisfies $1 \le y \le n$. We call such labeling a \textit{consistent} labeling. In such cases, the labeling is indeed a $pq$-distance magic labeling.
\begin{figure}[ht]
 \begin{center}
        \subfigure[$G_1$]{
            %\label{fig:}
            \includegraphics[width=0.2\textwidth]{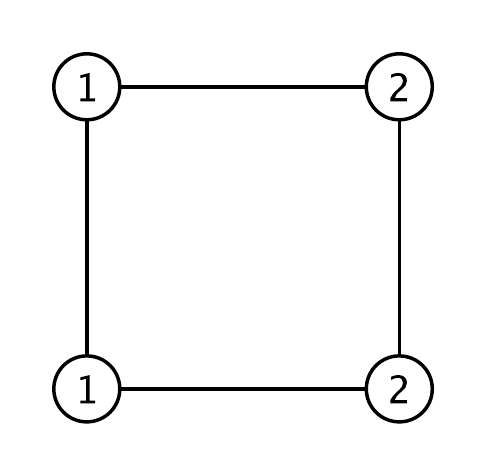}
        }
        \subfigure[$G_2$]{
           %\label{fig:}
           \includegraphics[width=0.2\textwidth]{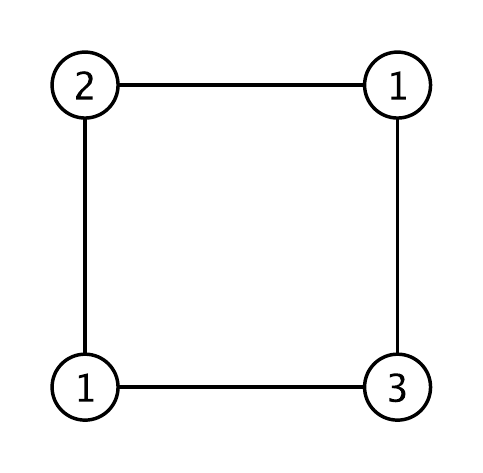}
        }
        \subfigure[$G_3$]{
           %\label{fig:}
           \includegraphics[width=0.2\textwidth]{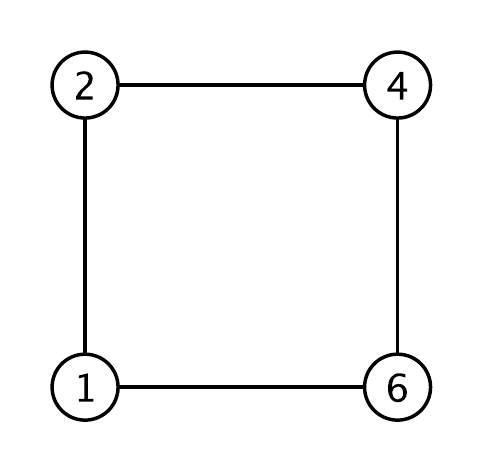}
        }
\end{center}
\caption{$p$-distance magic labelings of $C_4$}
\label{fig:crt}
\end{figure}
We state the following proposition without proof, as its proof is quite similar to Theorem \ref{th:crt}.
\begin{proposition}
If $G$ is a $p$-distance magic as well as a $q$-distance magic graph on $n$ vertices for some relatively prime integers $p$ and $q$ such that $pq > n$ and if the new labeling obtained as described in the proof of Theorem \ref{th:crt} is consistent then the graph $G$ is $pq$-distance magic. 
\end{proposition}

\section{Spectra of the linear operators related to distance magic graphs} \label{sec: matrices}

For a given graph $G$, let $\mu_1, \mu_2, \dots, \mu_t$ be distinct eigenvalues of its adjacency matrix $A$. We call $\mu$ an eigenvalue of $G$ whenever it is an eigenvalue of its adjacency matrix $A$. A graph is considered \textit{singular} if $0$ is the eigenvalue of $G$ and it is not singular otherwise. We always take eigenvalues in non-increasing order $\mu_1 \ge \mu_2 \ge \dots \ge \mu_t$. The largest eigenvalue that is, $\mu_1$, is called \textit{index} of $G$. The eigenvalue is called \textit{simple} if its algebraic multiplicity is $1$. As the adjacency matrix $A$ of a graph $G$ is real symmetric, it is diagonalisable. Therefore, its eigenvectors form an orthonormal basis of $\mathbb{R}^n$. Let $\mathcal{E}(\mu_i)$ be the eigenspace corresponding to the eigenvalue $\mu_i$. For $a \in \mathbb{R}^n$, let $P_i(a) = \sum_{b} \langle a, b \rangle b$, where $b$ is an orthonormal eigenvector corresponding to the eigenvalue $\mu_i$ and $\langle,\rangle$ is a standard inner product on $\mathbb{R}^n$. Let $\n$ be a column vector whose all entries are $1$. The \textit{main angles} of $G$, denoted by $\beta_i$ are given by $\beta_i = \frac{\left\|P_i \n \right\|}{\sqrt{n}}$. Let $J_n$ be an $n \times n$ matrix whose all entries are $1$.

Let $G$ be a distance magic graph on $n$ vertices $v_1, v_2, \dots, v_n$ with distance magic labeling $f$ and the magic constant $k$. Let $x =\left(f(x_1),f(x_2), \dots, f(x_n)\right)^{\top}$. Define a transformation $T:\mathbb{R}^n \rightarrow \mathbb{R}^n$ by $T(e_i) = e_i + e_{i+1} + \dots + e_n$, where $e_1, e_2, \dots, e_n$ is a standard basis of $\mathbb{R}^n$. Then, for some suitable permutation matrix $P$, we have $x = PT\n$. We assume these notations throughout this section.

\smallskip
The following theorem gives the necessary and sufficient condition for a graph to be distance magic in terms of the eigenvalues of the linear transformation $APT$.
\begin{theorem} \label{th: apt}
A graph $G$ is distance magic if and only if $\n$ is an eigenvector of $APT$ corresponding to the eigenvalue $k$ for some suitable permutation matrix $P$.
\end{theorem}
\begin{proof}
Let $G$ be the distance magic graph. Then, we have $Ax = k\n$. Choose $P$ such that $PT\n = x$. This implies $APT\n = Ax = k\n$ as required. Conversely, suppose that $APT$ has an eigenvalue $k$ corresponding to eigenvector $\n$. Then $APT\n = k\n$. Taking $x = PT\n$, we get $Ax = k\n$. This proves that $G$ is a distance magic graph.
\end{proof}

It is known that odd regular graphs can not be distance magic\cite{dm_miller}. In $2$-regular graphs, only disjoint unions of $C_4$ is distance magic\cite{dm_miller}. Distance magic labeling of $C_4$ is shown in Figure \ref{fig: dmp3&c4}. Miklavi\v{c} et al. \cite{dm_4reg_miklavic, dm_6reg_miklavic} gave a characterisation of $4$-regular and $6$-regular graphs. They proved the following general result based on a spectrum of a graph, giving necessary and sufficient conditions for the even regular graph to be distance magic.

\begin{lemma} \cite{dm_4reg_miklavic}  \label{th: dm_even} 
Let $G$ be an even regular graph of order $n$. Then $G$ is distance magic if and only if $0$ is an eigenvalue of $G$ and there exists an eigenvector $w$ for the eigenvalue $0$ with the property that a certain permutation of its entries results in the arithmetic sequence
\begin{equation} \nonumber
\frac{1-n}{2}, \frac{3-n}{2}, \frac{5-n}{2}, \dots, \frac{n-1}{2}.
\end{equation}
In particular, if $G$ is distance magic, then $0$ is an eigenvalue for $G$, and there exists a corresponding eigenvector, all of whose entries are pairwise distinct.
\end{lemma}

Let $G$ be a graph with vertex set $\{v_1, v_2, \dots, v_n\}$. Let $D = diag(deg(v_1), deg(v_2), \dots, deg(v_n))$ be a degree diagonal matrix of $G$. Laplacian matrix of a graph $G$ is given by $L = D-A$. Now we prove another simple necessary and sufficient condition for connected even regular graphs to be distance magic in terms of the eigenvalues of the matrix $L^{2} + A^{2}$. 

\begin{theorem}
Let $G$ be a connected $r$-regular graph, where $r$ is even. A graph $G$ is distance magic graph with distance magic labeling $f$ if and only if $(L^2 + A^2){x} = r^2 {x}$, where $x = (f(x_1), \dots, f(x_n))^{\top}$.
\end{theorem}
\begin{proof}
Let $G$ be connected $r$-regular distance magic graph with distance magic labeling $f$ and magic constant $k$. Then $A{x} = k\n$. We know $L = D - A$. Hence,
\begin{equation} \label{lap}
    L^{2} + A^{2} = D^{2} - 2rA + 2A^{2} =  r^{2}I - 2rA + 2A^{2}.
\end{equation}
Then
\begin{align*}
    (L^{2} + A^{2})x &= (r^{2}I - 2rA + 2A^{2})x\\
                 &= r^{2}x - 2rk\n + 2rk\n\\
                 &= r^{2}x.
\end{align*}
Conversely suppose that $(L^{2} + A^{2})x = r^{2}x$. Then, after expanding the left-hand side as in Equation (\ref{lap}) and canceling out the terms, we obtain $A^{2}x = rAx$, that is, $(A-rI) Ax = 0$. Now, by the regularity of $G$, the sum of entries of each row in $A-rI$ is $0$, and by connectedness of $G$, $r-1$ is an eigenvalue of $A-rI$ of multiplicity $1$. Therefore, $\{\n\}$ forms a basis for $\mathcal{E}(r-1)$. Note that each entry in $Ax$ is an integer. Therefore, we must have $Ax = k\n$ for some positive integer $k$. Therefore, $G$ is a distance magic graph.
\end{proof}

Now we prove that there are infinitely many singular distance magic graphs. Recall that the \textit{cone cover} of a graph $G$, denoted by $K_1 \triangledown G$, is a graph obtained by adding a vertex to $G$ and joining it to all other vertices of $G$. Kamatchi \cite{kamatchit} characterised the distance magic graphs $G$ with $\Delta(G) = n-1$. 

\begin{theorem} \cite{kamatchit} \label{cone}
Let $G$ be any graph of order $n$ with $\Delta = n - 1$. Then $G$ is a distance magic graph if and only if $n$ is odd and $G$ is isomorphic to the cone cover $K_1 \triangledown (K_{n-1}-M)$ of $K_{n-1}-M$, where $M$ is a perfect matching in $K_{n-1}$.
\end{theorem}

\begin{theorem} \label{th: singular}
For any positive integer $n \ge 3$, there is a singular distance magic graph of order $n$.
\end{theorem}
\begin{proof}
Let $G$ be a graph of order $n$. 

\noindent Case i: Let $n \ge 4$ be even and let $V(G) = \{1,2, \dots, n\}$. We may assume that the label of the vertex $i$ is $i$. For each $1 \le i \le n$, let $N(i) = V(G) - \{i, n+1-i\}$. Hence $w(i) = \frac{n(n+1)}{2} - (n+1)$ for all $i = 1,2, \dots, n$. Therefore, we conclude that $G$ is a distance magic graph. Also, by construction, it is clear that $N(1) = N(n)$. Therefore, the rows corresponding to $1$ and $n$ in the adjacency matrix of $G$ are identical. This proves that $G$ must be singular.\\
Case ii: Let $n \ge 3$ be odd. For $n = 3$, we know that $P_3$ is distance magic (see Figure \ref{fig: dmp3&c4}) and it is a singular graph. Suppose $n > 3$. Then there is a distance magic graph $G = K_{n-1}-M + K_{1}$, where $M$ is a perfect matching in $K_{n-1}$. If $uv \in M$ then $N_{K_{n-1}}(u) = N_{K_{n-1}}(v)$. Therefore, $N_{G}(u) = N_{G}(v)$. Proving that $G$ is singular.
\end{proof}

\smallskip
A graph is said to \textit{integral} if all of its eigenvalues are integers and it is not integral otherwise. We shortly prove that the cone cover of $K_{n-1}-M$ for odd $n$ is not integral.

We will calculate the characteristic polynomial of such graphs. There is a relation between the characteristic polynomial of $G$ and its cone cover in terms of the eigenvalues and the main angles of $G$:

\begin{theorem}\cite{cvetkovic} \label{conepoly}
The cone cover of $G$ has characteristic polynomial $P_{K_1 \triangledown G}(y) = P_{G}(y) \left(y - \sum_{i=1}^{k} \frac{n \beta_{i}^2}{y - \mu_i} \right)$, where $\mu_i$ and $\beta_j$ are eigen values and main angles of $G$ respectively.
\end{theorem}

Therefore, to calculate the characteristic polynomial of $K_1 \triangledown (K_{n-1}-M)$, we need the information about the spectra and main angles of $K_{n-1}-M$.

\begin{proposition} \label{conespectra}
Let $n \ge 4$ be even, and let $K_n$ be the complete graph on the $n$ vertices $1,2, \dots, n$. Let $M$ be a perfect matching in $K_n$ such that if $e \in M$ then $e = (i, i + \frac{n}{2})$. Then 
\begin{equation*}
\mathrm{spec}(K_n - M) =
    \begin{pmatrix}
    n-2 & 0 & -2\\
    1 & \frac{n}{2} & \frac{n}{2} - 1
\end{pmatrix}.
\end{equation*}
\end{proposition}

\begin{proof}
Deletion of a maximum matching in a complete graph $K_n$ $(n \ge 4)$ does not disconnect the graph. Therefore $G = K_{n}-M$ is a connected $(n-2)$-regular graph. Therefore, $(n-2)$ is a simple eigenvalue of $G$. Also, in $G$, $N(i) = N(i + \frac{n}{2})$. Therefore, the corresponding rows in the adjacency matrix are identical, and hence, the eigenvalue $0$ has multiplicity $\frac{n}{2}$. All $\frac{n}{2}$ eigenvectors in the basis of $\mathcal{E}(0)$ are given by $v_j$, where for each $1 \le j \le \frac{n}{2}$, $i$th coordinate of the vector $v_{j}$ is given by
\begin{align*}
    (v_j)_{i} = 
    \begin{cases}
    -1 &\text{if } i = j\\
    1 &\text{if } i = j + \frac{n}{2}\\
    0 &\text{otherwise}.
    \end{cases}
\end{align*}
For each $2 \le j \le (\frac{n}{2})$, consider a vector $w_{j}$ whose $i$th coordinate is given by
\begin{equation*}
(w_{j})_{i} = \begin{cases}
                    -1 &\text{if } i = 1, \frac{n}{2}+1\\
                    1 &\text{if } i = j, \frac{n}{2}+j\\
                    0 &\text{otherwise}.
                    \end{cases}     
\end{equation*}
It is easy to verify that these $\frac{n}{2}-1$ vectors $w_j$ are linearly independent and are the eigenvectors corresponding to eigenvalue $-2$. Therefore, the multiplicity of $-2$ is at least $\frac{n}{2}-1$. We already found $\frac{n}{2}+1$ other eigenvalues of $G$. Therefore algebraic multiplicity of the eigenvalue $-2$ must be $\frac{n}{2}-1$. This completes the proof.
\end{proof}

From the spectra of the graph $K_{n}-M$ as discussed previously, we can write its characteristic polynomial, $P_{K_{n}-M}(y) = y^{\frac{n}{2}}(y+2)^{\frac{n}{2}-1}(y-n+2)$. Now we calculate the main angles $\beta_i$ of $K_{n}-M$.\\

\smallskip
 Let $y = (y_1, y_2, \dots, y_n)^{\top} \in \mathbb{R}^{n}$. Recall that to calculate $P_i(y)$ we take orthonormal vectors from $\mathcal{E}(\mu_i)$. The only eigenvector corresponding to eigenvalue $(n-2)$ is $\n$. Therefore,
\begin{align*}
P_1(y) &= \left\langle y, \frac{\n}{\sqrt{n}} \right \rangle \frac{\n}{\sqrt{n}}\\
&= \frac{1}{n} (y_1 + y_2 + \dots + y_n,y_1 + y_2 + \dots + y_n, \dots, y_1 + y_2 + \dots + y_n).
\end{align*}
Therefore, the matrix representation of $P_1$ with respect to the standard basis of $\mathbb{R}^n$ is the matrix $J_n$. Therefore, $\left\|P_1 \n \right\| = 1$ and hence $\beta_1 = \frac{1}{\sqrt{n}}$. Now for any $v \in \mathcal{E}(0)$, $\left\| v \right\| = \sqrt{2}$.
\begin{align*}
 P_2(y) 
 &= \frac{1}{2} \sum_{v \in \mathcal{E}(0)} \langle y, v \rangle v\\
 &= \frac{1}{2} \sum_{j=1}^{\frac{n}{2}} (-y_{1} + y_{j+\frac{n}{2}}) v\\
 &= \frac{1}{2} (y_1 - y_{1+\frac{n}{2}}, y_2 - y_{2 + \frac{n}{2}}, \dots, y_{\frac{n}{2}} - y_{\frac{n}{2}+\frac{n}{2}}, - y_1 + y_{1+\frac{n}{2}}, - y_2 + y_{2 + \frac{n}{2}}, \dots, - y_{\frac{n}{2}} + y_{\frac{n}{2}+\frac{n}{2}}).
\end{align*}
Observe that the sum of elements in each row of the matrix representation of $P_2$ with respect to the standard basis of $\mathbb{R}^n$ is $0$. Therefore $\left\| P_2 \n \right\| = 0$ and hence $\beta_2 = 0$.
For any $v \in \mathcal{E}(-2)$, $\left\|v \right\| = \frac{1}{2}$.
\begin{align*}
 P_3(y)
 &= \frac{1}{4}\sum_{w \in \mathcal{E}(-2)} \langle y, w \rangle w\\
 &= \frac{1}{4}\sum_{j=2}^{\frac{n}{2}} (-y_{1} + y_{j} - y_{1+\frac{n}{2}} + y_{j+\frac{n}{2}}) w_j = w,
\end{align*}
where $i$th coordinate of $w$ is given by
\begin{equation*}
    w = \begin{cases}
                (\frac{n}{2}-1) y_1 - y_2- \dots - y_{\frac{n}{2}} + (\frac{n}{2}-1) y_{\frac{n}{2}+1} - y_{\frac{n}{2}+2} - \dots - y_n & \quad \text{ if } i = 1, \frac{n}{2}+1\\
                -y_1 + y_{i} - y_{\frac{n}{2}+1} + y_{\frac{n}{2}+i} & \quad \text{ otherwise}
              \end{cases}
\end{equation*}
Again, observe that the sum of elements in each row of matrix representation of $P_3$ with respect to the standard basis of $\mathbb{R}^n$ is $0$. Therefore $\left\| P_3 \n \right\| = 0$ and hence $\beta_3 = 0$.\\

If $n=3$, then $K_1 \triangledown (K_{n-1}-M) \cong P_3$ and 

\begin{equation*}
\mathrm{spec}(P_3) = 
\begin{pmatrix}
\sqrt{2} & 0 & -\sqrt{2}\\
1 & 1 & 1
\end{pmatrix}.  
\end{equation*} 

Therefore, $P_3$ is not integral. Let $n \ge 5$ be an odd integer.
By Theorem \ref{conepoly}, we have,
\begin{equation} \label{eq:conepoly} 
P_{K_1 \triangledown (K_{n}-M)}(y) = P_{K_{n}-M}(y) \left(y - \sum_{i=1}^{k} \frac{n \beta_{i}^2}{y - \mu_i} \right).
\end{equation}
By Proposition \ref{conespectra}, distinct eigenvalues of $K_{m}-M$, where $m = n-1$ are $m-2, 0$ and $-2$. Also the main angles of $K_{m}-M$ are, $\beta_{1} = \frac{1}{\sqrt{m}}, \beta_2 = \beta_3 = 0$ and $P_{K_{m}-M}(y) = y^{\frac{m}{2}}(y+2)^{\frac{m}{2}-1}(y-m+2)$. Therefore, the Equation (\ref{eq:conepoly}) becomes
\begin{align*}
P_{K_1 \triangledown (K_{m}-M)} (y)
&= y^{\frac{m}{2}}(y+2)^{\frac{m}{2}-1}(y-m+2) \left(y - \frac{m \left(\frac{1}{\sqrt{m}}\right)^2}{y-(m-2)} \right)\\
&= y^{\frac{m}{2}}(y+2)^{\frac{m}{2}-1} (y^2 + (2-m)y -1).
\end{align*}
\begin{proposition}
If $G$ is a distance magic graph of order $n$ with $\Delta(G) = n-1$, then $G$ is not integral.
\end{proposition}  
\begin{proof}
We know the if $G$ is distance magic graph of order $n$ with $\Delta(G) = n-1$ then $n$ is odd and $G$ is isomorphic to a cone cover of $K_{m}-M$, where $m = n-1$ and its characteristic polynomial is $P_{K_1 \triangledown (K_{m}-M)} (y) = y^{\frac{m}{2}}(y+2)^{\frac{m}{2}-1} (y^2 + (2-m)y -1)$. Consider the factor $y^{\frac{m}{2}}(y+2)^{\frac{m}{2}-1}$. If $n = 3$, we know that $P_3$ is not integral as $\sqrt{2}$ is its eigenvalue. Let us assume that $n \ge 5$. With $c = m-2 \ge 2$ even, we rewrite the polynomial $y^2 - (m-2)y - 1$ as $y^2 - cy - 1$. The roots of this polynomial are: $y = \frac{c \pm \sqrt{c^2 + 1}}{2}$. Again, since $c$ is even, we put $c = 2d$. Then the roots become $y = d \pm \sqrt{4d^2 + 1}$. But $4d^2+1$ is not a perfect square. If $4d^{2}+1 = a^{2}$ for some integer $a$, then we get $(2d-a)(2d+a) = 1$, which is not possible as $a$ and $d$ both are integers. Therefore, the cone cover of $G$ has a non-integer eigenvalue as required. This completes the proof.
\end{proof}

Up to the isomorphism, there are only $23$ distance magic graphs of order up to $10$ \cite{dm_algo_fuad, dm_algo1}. It can be manually checked that in any distance magic graph $G$ of order up to $10$, there exist vertices $u,v$ such that $N(u) = N(v)$. Therefore, the rows in the adjacency matrix of $G$ corresponding to the vertices $u$ and $v$ are identical. From this, we conclude that such a graph $G$ must be singular. Therefore, every distance magic graph of an order up to $10$ is singular. The nonsingular distance magic graph of the smallest order (on the $11$ vertices) is shown in Figure \ref{fig: ndm}. Also, as discussed earlier, there is at least one singular distance magic graph of order $n$ for each $n \ge 3$. These observations shifts our attention to the generalised inverse of the adjacency matrix of distance magic graphs.

\begin{figure}[ht]
\centering
\begin{tikzpicture}
% Vertices
  \foreach \i in {1,...,11}
    \node[circle, draw, inner sep=2pt] (\i) at ({360/11 * (\i-1)}:2) {\i};
%edges
  \draw (1) -- (2);
  \draw (1) -- (5);
  \draw (1) -- (6);
  \draw (1) -- (7);
  \draw (1) -- (11);

  \draw (2) to[bend left = 10] (4);
  \draw (2) -- (7);
  \draw (2) -- (9);
  \draw (2) -- (10);

  \draw (3) -- (6);
  \draw (3) -- (7);
  \draw (3) -- (8);
  \draw (3) -- (10);

  \draw (4) -- (5);
  \draw (4) -- (7);
  \draw (4) -- (8);
  \draw (4) -- (9);

  \draw (5) to[bend left = 10] (7);
  \draw (5) -- (8);
  \draw (5) -- (11);

  \draw (6) -- (7);
  \draw (6) -- (9);
  \draw (6) -- (11);

  \draw (7) -- (10);

  \draw (8) -- (9);
  
  \draw (8) to[bend left = 10] (10);

  \draw (9) to[bend left = 12] (11);

  \draw (10) -- (11);
\end{tikzpicture}    
\caption{Non-singular distance magic graph of the smallest order}
\label{fig: ndm}
\end{figure}
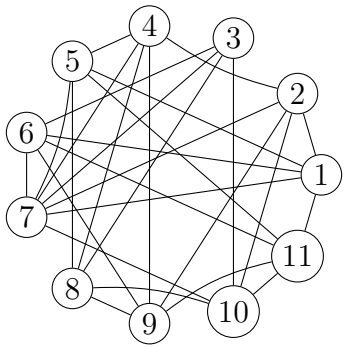

\smallskip
A matrix $B^{\dagger} \in \mathbb{R}^{n \times m}$ is said to be the \textit{Moore-Penrose} (or the \textit{generalised}) \textit{inverse} of a matrix $B \in \mathbb{R}^{m \times n}$ if $BB^{\dagger}B = B$, $B^{\dagger}BB^{\dagger} = B^{\dagger}$, $(BB^{\dagger})^{*} = BB^{\dagger}$ and $(B^{\dagger}B)^{*} = B^{\dagger}B$. We give another necessary condition for a graph to be distance magic in terms of the generalised inverse of its adjacency matrix. Recall that a \textit{doubly stochastic matrix} is a matrix whose each row and each column sum to one.
\begin{theorem}
If $G$ is distance magic, then $AA^{\dagger}$ is doubly stochastic, where $A^{\dagger}$ is the generalised inverse of an adjacency matrix $A$ of $G$.
\end{theorem}
\begin{proof}
Let $G$ be the distance magic graph with magic constant $k$. Then $Ax = k\n$.
	\begin{align*}
		&\implies AA^{\dagger}Ax = kAA^{\dagger} \n\\
		&\implies Ax = kAA^{\dagger} \n\\
		&\implies k\n = kAA^{\dagger} \n\\
		&\implies \n = AA^{\dagger} \n.
	\end{align*}
This shows that each row in $AA^{\dagger}$ sum to $1$. Since $AA^{\dagger}$ is symmetric, column sums are also $1$. This proves the theorem.
\end{proof}
The converse of the above theorem is not true. For example, take $G = K_{1,3}$ and let $A$ be its adjacency matrix. Then
\[
\begin{array}{cc}
\begin{subarray}{c}
A^{\dagger} = 
\begin{pmatrix}
    \frac{1}{3} & 0 & 0 & 0\\
    \frac{1}{3} & 0 & 0 & 0\\
    \frac{1}{3} & 0 & 0 & 0\\
    0 & \frac{1}{3} & \frac{1}{3} & \frac{1}{3}
\end{pmatrix}
\end{subarray}
\hspace{1 cm}
\begin{subarray}{c}
AA^{\dagger} = 
\begin{pmatrix}
    1 & 0 & 0 & 0\\
    0 & \frac{1}{3} & \frac{1}{3} & \frac{1}{3} \\
    0 & \frac{1}{3} & \frac{1}{3} & \frac{1}{3} \\
    0 & \frac{1}{3} & \frac{1}{3} & \frac{1}{3}
\end{pmatrix}

\end{subarray}

\end{array}
\]
and $\n$ is an eigenvector of $AA^{\dagger}$ corresponding to eigenvalue $1$, but $G$ is not distance magic. 

\section{Action of the automorphism group of the graph on its labeling set} \label{sec: action}
Let $G$ be a graph on $n$ vertices, and let $Aut(G)$ be the group of automorphisms of $G$. Let $\mathcal{M}(G) = \{f : f \text{ is distance magic labeling of } G\}$ be set of distance magic labelings of $G$. 

\begin{theorem} \label{th: bound}
A distance magic graph \(G\) has at least \(|Aut(G)|\) distance magic labelings.
\end{theorem}
\begin{proof}
Let $G$ be a graph with $V(G) = \{1,2, \dots, n\}$ and  $\mathcal{M}$ be a set of all distance magic labelings of a graph $G$. Define action of $Aut(G)$ on $\mathcal{M}$ by $(\sigma \cdot f)(i) = f\sigma^{-1}(i)$, for all $f \in \mathcal{M}$ and for all $i \in V$. First, we will show that $\sigma \cdot f \in \mathcal{M}$ for all $f$ and all $\sigma$. For $v \in V$, let $N(v) = \{u_1, u_2, \dots, u_t\}$. Since $f \in \mathcal{M}$, $w_{f}(v) = k$ and since $\sigma \in Aut(G)$, $u_i \in N(v)$ implies $\sigma^{-1}(u_i) \in N(\sigma^{-1}(v))$. 
\begin{align*}
w_{ \sigma \cdot f}(v) 
&= \sum_{u \in N(v)} \sigma \cdot f (u)\\
&=  \sum_{u \in N(v)}  f \sigma^{-1} (u)\\
&= \sum_{x \in N(\sigma^{-1}(v))}f(x)\\
&= k.
\end{align*}
Now it is easy to check that $Aut(G)$ acts on $\mathcal{M}$ under the above action. For $f \in \mathcal{M}$, let $\mathcal{S}_f$ denotes the stabiliser of $f$ and let $\mathcal{O}_f$ denotes the orbit of $f$. Let $\sigma \cdot f = f$ for some $\sigma \in Aut(G)$. Then $f\sigma^{-1} = f$ i.e. $f\sigma^{-1}(x) = f(x)$ for all $x \in V$. Since both $f$ and $\sigma$ are bijective functions, we conclude that $\sigma \equiv 1$, the identity automorphism. This proves that $|\mathcal{S}_f| = 1$. Therefore by orbit stabiliser theorem, $|Aut(G)| = |\mathcal{O}_f|$. As orbits of elements of $\mathcal{M}$ partitions $\mathcal{M}$, each of which has same size $|Aut(G)|$, $|Aut(G)|$ divides $|\mathcal{M}|$. This completes the proof.
\end{proof}

\begin{figure}[h]
\centering
\begin{tikzpicture}
    \node[circle, draw, inner sep=2pt] (a) at (0,1) {$1$};
    \node[circle, draw, inner sep=2pt] (b) at (2,1) {$3$};
    \node[circle, draw, inner sep=2pt] (c) at (4,1) {$2$};

    \node[circle, draw, inner sep=2pt] (d) at (6,0) {$1$};
    \node[circle, draw, inner sep=2pt] (e) at (6,2) {$2$};
    \node[circle, draw, inner sep=2pt] (f) at (8,2) {$4$};
    \node[circle, draw, inner sep=2pt] (g) at (8,0) {$3$};

    \draw (a) -- (b) -- (c);
    \draw (d) -- (e) -- (f) -- (g) -- (d);
\end{tikzpicture}
\caption{Distance magic labeling of graphs $P_3$ and $C_4$}
\label{fig: dmp3&c4}
\end{figure}
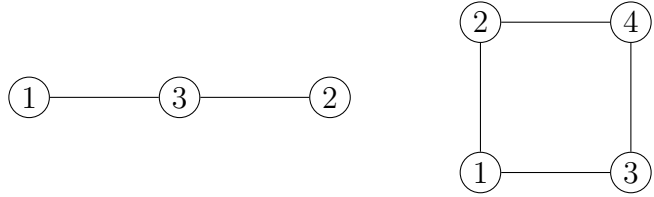

For a given distance magic graph $G$, we know that $|Aut(G)| \le |\mathcal{M}|$. For simplicity, for a given graph $G$, we will denote its vertex labeling as $(f(v_1), f(v_2),\dots, f(v_n))$ or simply $(f(v_1)f(v_2) \dots f(v_n))$ when there is no confusion. Now, we give some examples showing that the lower bound obtained on the size of distance magic labelings $\mathcal{M}$ is tight for several graphs. The equality is observed for the distance magic graphs $P_3$ and $C_4$. The distance magic labelings of $P_3$ and $C_4$ with magic constants $3$ and $5$, respectively, are shown in Figure \ref{fig: dmp3&c4}.

\smallskip
Recall that $|Aut(P_3)| = 2$. Let $P_3 : v_1,v_2,v_3$ be $v_1,v_3$-path. It is easy to verify that $\mathcal{M}(P_3) = \{(132), (231)\}$ is the complete set of distance magic labelings of $P_3$. Hence, $|Aut(P_3)| = |\mathcal{M}(P_3)|$. For a cycle $C_4 : v_1, v_2, v_3, v_4, v_1$; $\mathcal{M}(C_4) = \{(1243), (3124), (4312), (2431), (2134), (1342), (3421), (4213)\}$ and $|Aut(C_4)| = 8$. Therefore, $|Aut(C_4)| = |\mathcal{M}(C_4)|$.
\begin{figure}[ht]
\centering
\begin{tikzpicture}[]
  % Vertices
  \foreach \coord/\name/\label [count=\i] in {(0,0)/1/4, (0,2)/2/2, (2,2)/3/3, (2,0)/4/5, (3,2)/7/1, (3,1)/6/7, (3,0)/5/6, (5,0)/8/1, (5,2)/9/2, (7,2)/10/6, (7,0)/11/5, (8,2)/12/3, (8,1)/13/7, (8,0)/14/4, (10,0)/15/1, (10,2)/16/3, (12,2)/17/6, (12,0)/18/4, (13,2)/19/2, (13,1)/20/7, (13,0)/21/5}
  %labels
   \node[circle, draw, inner sep=2pt] (\name) at \coord {\label};
   %names
   \node[below = 10pt] at (1,0) {$M_1$};
   \node[below = 10pt] at (6, 0) {$M_2$};
   \node[below = 10pt] at (11, 0) {$M_3$};
  % Edges
  \draw (1) -- (2) -- (3) -- (4) -- (1);
  \draw (7) -- (6) -- (5);
  \draw (8) -- (9) -- (10) -- (11) -- (8);
  \draw (12) -- (13) -- (14);
  \draw (15) -- (16) -- (17) -- (18) -- (15);
  \draw (19) -- (20) -- (21);
\end{tikzpicture}
\caption{Various distance magic labeling schemes of $P_3 \cup C_4$}
\label{fig: dmp3c4}
\end{figure}
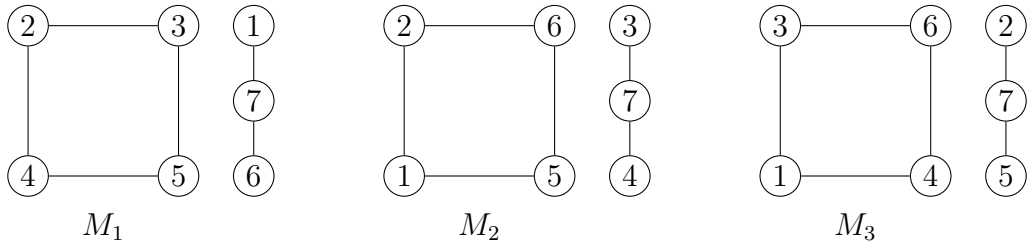

For the distance magic graph $P_3 \cup C_4$, the inequality is strict. We know that $|Aut(P_3 \cup C_4)| = 16$. From each of the distance magic labeling schemes $M_1$, $M_2$ and $M_3$, shown in Figure \ref{fig: dmp3c4}, $16$ different distance magic labelings of $P_3 \cup C_4$ can be obtained using automorphisms of $P_3 \cup C_4$. Therefore, $P_3 \cup C_4$ has at least $48$ different distance magic labelings.

\section{Conclusion and further directions in research}
The graph-theoretic properties and the spectrum of the class `distance magic graphs are unexplored. In this paper, we have studied several graph-theoretic properties of this class.
We raise the following problems:

\begin{problem} \label{prob:2}
Given a graph $G$, which is not distance magic, characterise all the primes $p \le \frac{n^{2}-1}{2}$ such that $G$ is not $p$-distance magic graph.
\end{problem}

\begin{problem}
Characterise the distance magic graphs in terms of the spectra of the matrices related to the graphs.
\end{problem}
\begin{problem}
Determine the class of singular as well as integral distance magic graphs.
\end{problem}
\begin{problem}
Characterisation of distance magic graphs $G$ for which $|Aut(G)| = |\mathcal{M}(G)|$.
\end{problem}
Similar bounds as given in Theorem \ref{th: bound} can be obtained for other variants of graph labelings: for example, say \textit{ antimagic labeling} (a graph $G$ is said to be antimagic if all of its vertex weights under the edge labeling $f: E(G) \to \{1,2, \dots, |E(G)|\}$ are distinct. For more details see \cite{dm_survey_gallian}).

\begin{problem}
Characterise the distance magic graphs of order $n$ with maximum degree $(n-2)$. 
\end{problem}
\begin{problem}
Given a graph $G$ and a positive integer $p$, study the computational complexity of the decision problem: whether $G$ is $p$-distance magic?
\end{problem}

\bibliographystyle{plain} 
\bibliography{refs}

\end{document}